\documentclass[12pt,leqno]{amsart}
\usepackage{enumerate}
\usepackage{amsrefs}
\usepackage{amsfonts, amsmath, amssymb, amscd, amsthm, bm, cancel}
\usepackage{url}
\usepackage{graphicx}
\usepackage[
linktocpage=true,colorlinks,citecolor=magenta,linkcolor=blue,urlcolor=magenta]{hyperref}
\usepackage{multicol}
\usepackage{comment}
\usepackage[margin=1in]{geometry}
\makeatletter
\@namedef{subjclassname@2020}{\textup{2020} Mathematics Subject Classification}
\makeatother
\theoremstyle{plain} \newtheorem{thm}{Theorem}[section]
\theoremstyle{plain} \newtheorem{prop}[thm]{Proposition}
\theoremstyle{plain} 
\theoremstyle{definition} 
\theoremstyle{plain} \newtheorem{cor}[thm]{Corollary}
\theoremstyle{plain} \newtheorem{lem}[thm]{Lemma}
\theoremstyle{plain} 
\theoremstyle{remark} \newtheorem{rmk}[thm]{Remark}
\theoremstyle{plain} \newtheorem{conj}{Conjecture}
 
 \newtheorem*{ack}{Acknowledgments}

 \numberwithin{equation}{section}
 \allowdisplaybreaks

\newcommand{\divg}{\mathrm{div}}

\newcommand{\bS}{\mathbb{S}}

\newcommand{\la}{\langle}
\newcommand{\ra}{\rangle}

\begin{document}
\title[Weinstock inequality in hyperbolic space]{Weinstock inequality in hyperbolic space}

\author{Pingxin Gu}
\address{Department of Mathematical Sciences, Tsinghua University, Beijing 100084, P.R. China}
\email{\href{mailto:gpx21@mails.tsinghua.edu.cn}{gpx21@mails.tsinghua.edu.cn}}

\author{Haizhong Li}
\address{Department of Mathematical Sciences, Tsinghua University, Beijing 100084, P.R. China}
\email{\href{mailto:lihz@tsinghua.edu.cn}{lihz@tsinghua.edu.cn}}

\author{Yao Wan}
\address{Department of Mathematics, The Chinese University of Hong Kong, Shatin, Hong Kong, P.R. China}
\email{\href{mailto:yaowan@cuhk.edu.hk}{yaowan@cuhk.edu.hk}}



\begin{abstract}
In this paper, we establish the Weinstock inequality for the first non-zero Steklov eigenvalue on star-shaped mean convex domains in hyperbolic space $\mathbb{H}^n$ for $n \geq 4$.  In particular, when the domain is convex, our result gives an affirmative answer to Open Question 4.27 in \cite{CGGS24} for the hyperbolic space $\mathbb{H}^n$ when $n \geq 4$.
\end{abstract}

\maketitle

\section{Introduction}\label{sec-1}

One of the oldest topics in spectral geometry is shape optimization. We focus on the first non-zero eigenvalue of the Steklov eigenvalue problem introduced by Steklov in the early $20$th century. We refer readers to the notable surveys \cites{GP17,CGGS24}.

Let $\Omega$ be a compact Riemannian manifold with Lipschitz boundary $\partial\Omega$. The Steklov eigenvalue problem on $\Omega$ is given by
\begin{align*}
    \left\{\begin{array}{ll}
    \Delta u=0,&\text{in $\Omega$},\\
    \frac{\partial u}{\partial \nu}=\sigma(\Omega)u,&\text{on $\partial\Omega$},
\end{array}\right.
\end{align*}
where $\Delta$ is the Laplacian operator on $\Omega$, and $\nu$ is the outward unit normal of $\partial\Omega$. We may compute the first non-zero Steklov eigenvalue of $\Omega$ by
\begin{align}\label{def-steklov}
\sigma_1(\Omega)=\min\left\{\frac{\int_{\Omega}|\nabla u|^2 dv}{\int_{\partial\Omega}u^2 d\mu}:\ u\in H^1(\Omega)\backslash \{0\},\int_{\partial\Omega}u d\mu=0\right\},
\end{align}
where $dv$ is the volume element of $\Omega$ and $d\mu$ is the area element of $\partial\Omega$.

The earliest isoperimetric type inequality on $\sigma_1(\Omega)$ is the well-known Weinstock's inequality, which states that among simply-connected planar domains $\Omega$ with a fixed perimeter $|\partial\Omega|$, $\sigma_1(\Omega)$ is maximized by a disk, see \cites{Wei54,Wei54-2}. It is also possible to maximize $\sigma_1(\Omega)$ by fixing the volume $|\Omega|$, see \cite{Esc99}, and \cite{Bro01} for higher dimensions, where the connectedness assumption can be removed in both cases.

The simply-connected assumptions in Weinstock inequality cannot be removed, since the result fails for appropriate spherical shell $\Omega_{\epsilon}=B(1)\backslash \overline{B(\epsilon)}$ ($B(r)$ denotes the geodesic ball of radius $r$ centered at the origin), see for example \cite{GP17}. In higher dimensions $n\geq 3$, even for the class of contractible domains, the ball is not a maximizer, see \cite{FS19}. However, when considering convex domains, using the inverse mean curvature flow (IMCF), Bucur, Ferone, Nitsch and Trombetti \cite{BFNT21} proved that

\begin{thm}[\cite{BFNT21}]\label{Thm-BFNT-3.1}
Let $\Omega$ be a bounded convex domain in $\mathbb{R}^n$. Then
\begin{align}\label{eq-thmbfnt}
    \sigma_1(\Omega)|\partial\Omega|^{\frac{1}{n-1}}\leq \sigma_1(B)|\partial B|^{\frac{1}{n-1}},
\end{align}
where $B\subset \mathbb{R}^n$ is a ball, and equality holds if and only if $\Omega$ is a ball.
\end{thm}

Combining IMCF with elementary methods, Kwong and Wei \cite{KW23} proved sharp geometric inequalities involving three quantities, and extended Theorem \ref{Thm-BFNT-3.1} to a broader class of star-shaped mean convex domains. It is then natural to ask whether such a Weinstock type inequality holds in hyperbolic space, specifically

\begin{conj}[\cite{CGGS24}, Open Question 4.27]\label{conj}
Let $\Omega$ be a bounded convex domain in $\mathbb{H}^n$. Let $\Omega^*$ be a geodesic ball with $|\partial\Omega^*|=|\partial\Omega|$. Is it true that
\begin{equation}
    \sigma_1(\Omega)\leq\sigma_1(\Omega^*)
\end{equation}
and equality holds if and only if $\Omega$ is a geodesic ball?
\end{conj}

For $n=2$, the conjecture holds, as the Weinstock inequality is valid for the hyperbolic plane and, more generally, for compact surfaces, see \cite{FS11}. There are also many researches devoted in estimating general eigenvalues on general surfaces, see also \cites{GP12,KV14,YY17,Kar17}.

For higher dimensions, to the authors knowledge, the conjecture remains open. But if we turn to the case of fixing volume $|\Omega|$, then the geodesic ball is maximizer, see \cite{BS14}. When $\Omega$ is convex, this result can be a corollary of Conjecture \ref{conj} by means of isoperimetric inequality, see Remark \ref{rmk-2}.

In this paper, we prove the Weinstock inequality for smooth domains with star-shaped mean convex boundary in $\mathbb{H}^n$ for general dimension $n\geq 4$.

\begin{thm}\label{main thm}
Let $n\geq 4$, $\Omega$ be a smooth bounded domain in $\mathbb{H}^n$ with star-shaped mean convex boundary $\partial\Omega$, and $\Omega^{\ast}$ be a geodesic ball with $|\partial\Omega^*|=|\partial\Omega|$. Then
\begin{equation}\label{main-equ}
    \sigma_1(\Omega)\leq \sigma_1(\Omega^{\ast})
\end{equation}
and equality holds if and only if $\Omega$ is a geodesic ball.
\end{thm}

In particular, when the domain is convex, we provide an affirmative answer to Conjecture \ref{conj} for $n\geq 4$.

\begin{cor}
     Conjecture \ref{conj} is true for $n\geq 4$.
\end{cor}

\begin{rmk}\label{rmk-1}
Let $I_1(t)=|\partial B(t)|$ and $J(t)=\sigma_1(B(t))$ for $t>0$. Then (\ref{main-equ}) is equivalent to the following isoperimetric type inequality
\begin{equation}\label{main-equ2}
    \frac{\sigma_1(\Omega)}{J\circ I_1^{-1}(|\partial\Omega|)}\leq \left.\frac{\sigma_1(\Omega)}{J\circ I_1^{-1}(|\partial\Omega|)}\right|_{\text{$\Omega$ is a geodesic ball}}=1.
\end{equation}
\end{rmk}

Our method is inspired by the use of IMCF as presented in \cite{KW23} and by the well-known mass transplantation argument, as discussed in \cite{FL21}. The application of IMCF is not straightforward, as we must introduce a special auxiliary function $h$ in \eqref{def-h}. The evolution of this function under IMCF establishes a useful geometric inequality involving $\int_{\partial\Omega} g(r)^2d\mu$. By the same estimate, as a byproduct, we can extend Verma's isoperimetric inequality for the harmonic mean of the first $(n-1)$ non-zero Steklov eigenvalues \cite{Ver21} with prescribed area. We state it as the following Corollary \ref{main cor}, which can lead to Theorem \ref{main thm} again.

\begin{cor}\label{main cor}
Let $n\geq 4$, $\Omega$ be a smooth bounded domain in $\mathbb{H}^n$ with star-shaped mean convex boundary $\partial\Omega$, and $\Omega^{\ast}$ be a geodesic ball with $|\partial\Omega^*|=|\partial\Omega|$. Then
\begin{equation}
    \sum_{i=1}^{n-1}\frac{1}{\sigma_i(\Omega)}\geq \sum_{i=1}^{n-1}\frac{1}{\sigma_i(\Omega^{\ast})}
\end{equation}
and equality holds if and only if $\Omega$ is a geodesic ball.
\end{cor}

$\ $

The paper is organized as follows. In section \ref{sec-2}, we present basic properties of hyperbolic spaces, including the isoperimetric inequality and the first non-zero Steklov eigenvalue of geodesic balls. We also state the well-known theorems for center of mass and mass transplantation. In section \ref{sec-3}, we discuss the detailed properties of the volume-area ratio function $g$. In section \ref{sec-4}, we introduce the key auxiliary function $h$, get the estimate using inverse mean curvature flow, and prove Theorem \ref{main thm} and Corollary \ref{main cor}.

\begin{ack}
	The work was supported by NSFC Grant No. 12471047. We would like to thank Yong Wei for his valuable comments and suggestions.
\end{ack}

\section{Preliminaries}\label{sec-2}
We use the warped product to illustrate the hyperbolic space $(\mathbb{H}^n,g_{\mathbb{H}^n})=[0,\infty)\times_{\lambda(r)} \mathbb{S}^{n-1}$, where $\lambda(r)=\sinh r$. That is, after choosing an origin $O$ in $\mathbb{H}^n$, the Riemannian metric $g_{\mathbb{H}^n}$ on $\mathbb{H}^n$ is defined by
\[
g_{\mathbb{H}^n}=dr^2+\lambda^2(r)g_{\mathbb{S}^{n-1}},
\]
where $r$ denotes the geodesic distance from $O$, and $g_{\mathbb{S}^{n-1}}$ is the canonical metric on the unit sphere $\mathbb{S}^{n-1}$. We will also use the notation $\la \cdot,\cdot\ra$ to denote $g_{\mathbb{H}^n}$.

\subsection{The first non-zero Steklov eigenvalue of geodesic balls}$\ $

For $r> 0$, a geodesic ball $B(r)\subset \mathbb{H}^n$ centered at $O$ with radius $r$ is given by $\{(s,\theta)\in \mathbb{H}^n:\ 0\leq s\leq r\}\subset \mathbb{H}^n$, and its boundary is the geodesic sphere $S(r)=\partial B(r)=\{(s,\theta)\in \mathbb{H}^n:\ s=r\}$. We can directly compute the volume and area of $B(r)$ as
\begin{align}\label{eq-B}
|B(r)|=\int_0^r\int_{\mathbb{S}^{n-1}}\lambda^{n-1}(t)dtd\mu_{\mathbb{S}^{n-1}}=\omega_{n-1}\int_0^r\lambda^{n-1}(t)dt
\end{align}
and
\begin{align}\label{eq-S}
|S(r)|=\int_{\mathbb{S}^{n-1}}\lambda^{n-1}(r)d\mu_{\mathbb{S}^{n-1}}=\omega_{n-1}\lambda^{n-1}(r),
\end{align}
where $\omega_{n-1}$ denotes the area of $\mathbb{S}^{n-1}$. 

It is remarkable that the geodesic ball minimizes the area among bounded domains with the same volume,  which is known as the isoperimetric inequality in hyperbolic space.
\begin{thm}[\cite{Sch43}]\label{thm-isoperimetric inequality}
Let $\Omega\subset \mathbb{H}^n$ be a bounded domain and $B(R)$ be a geodesic ball with the same volume as $\Omega$. Then the area of $\Omega$ satisfies
\begin{align}\label{eq-isoperimetric}
|\partial\Omega|\geq |\partial B(R)|.
\end{align}
Furthermore, equality holds if and only if $\Omega$ is a geodesic ball.
\end{thm}

\begin{rmk}\label{rmk-2}
As noted in Remark \ref{rmk-1}, if we additionally denote $I_0(t)=|B(t)|$ for $t>0$, then \eqref{eq-isoperimetric} is equivalent to the following inequality
\begin{align}
\frac{|\Omega|}{I_0\circ I_1^{-1}(|\partial\Omega|)}\leq \left.\frac{|\Omega|}{I_0\circ I_1^{-1}(|\partial\Omega|)}\right|_{\text{$\Omega$ is a geodesic ball}}=1.
\end{align}
Combining this with (\ref{main-equ2}), we deduce
\begin{align}
    \frac{\sigma_1(\Omega)}{J\circ I_0^{-1}(|\Omega|)}\leq \left.\frac{\sigma_1(\Omega)}{J\circ I_0^{-1}(|\Omega|)}\right|_{\text{$\Omega$ is a geodesic ball}}=1.
\end{align}
\end{rmk}

An important function in hyperbolic space is the geodesic ball's volume-area ratio function $g$, defined as $g(r):=\frac{|B(r)|}{|S(r)|}$, i.e.
\begin{align}\label{eq-def-g}
g(r)=\frac{1}{\lambda^{n-1}(r)}\int_0^r\lambda^{n-1}(t)dt.
\end{align}
Using the standard method of separation of variables, the first non-zero Steklov eigenvalue of a geodesic ball is computed as follows:
\begin{prop}[\cite{BS14}]
Let $B(r)\subset \mathbb{H}^n$ be a geodesic ball centered at $O$ with boundary $S(r)$. Then the first non-zero eigenvalue $\sigma_1(B(r))$ of the Steklov eigenvalue problem \eqref{def-steklov} on $B(r)$ is given by
\begin{align}\label{eq-sigmabr}
\sigma_1(B(r))=\frac{\int_{B(r)}((g')^2+(n-1)\frac{g^2}{\lambda^2}) dv}{g(r)^2|S(r)|}.
\end{align}
\end{prop}

\subsection{Center of mass}$\ $

In order to study the upper bound of the first non-zero Steklov eigenvalue not only for geodesic balls but also for general domains, we need the following center of mass theorem. 

In \cite{AS96}, Aithal and Santhanam give a theorem for complete Riemannian manifold $M$, which also applies to measurable subset $A$ in hyperbolic space. Let $CA$ denote the convex hull of $A$. Let $\exp_q:T_qM\to M$ be the exponential map and let $X=(x_1,\ldots,x_n)$ be a system of normal coordinates centered at $q$. We identify $CA$ with $\exp_q^{-1}(CA)$ for each $q\in CA$, and denote $(g_{\mathbb{H}^n})_q(X,X)$ as $\|X\|_q^2$ for $X\in T_qM$. The theorem of center of mass with respect to the mass distribution function $G$ is stated as
\begin{thm}[\cite{BS14}]
Let $A\subset \mathbb{H}^n$ be a measurable subset and $G$ be a continuous function on $[0,\infty)$ that is positive on $(0,\infty)$. Then there exists a point $p\in CA$ such that
\[
\int_{A}G(\|X\|_p)X dv=0,
\]
where $X=(x_1,\ldots,x_n)$ is a normal coordinate system centered at $p$.
\end{thm}

By applying the center of mass theorem for $A=\partial\Omega$ with respect to $G=g$ and choosing test functions  $g\frac{x_i}{r}$ in the Rayleigh quotient \eqref{def-steklov}, we obtain
\begin{align*}
    \sigma_1(\Omega)\int_{\partial\Omega}g^2\frac{x_i^2}{r^2} d\mu\leq \int_{\Omega}|\nabla (g\frac{x_i}{r})|^2 dv,\qquad i=1,\ldots,n.
\end{align*}
Summing over $i$, we derive the upper bound of the first non-zero Steklov eigenvalue via $g$. For further details, see e.g. \cites{AV22, BS14, Ver21}.
\begin{prop}
Let $\Omega$ be a bounded domain in $\mathbb{H}^n$ with smooth boundary $\Sigma=\partial \Omega$. Then
\begin{align}\label{eq-upperbound-steklov}
\sigma_1(\Omega)\int_{\Sigma}g^2 d\mu\leq \int_{\Omega}((g')^2+(n-1)\frac{g^2}{\lambda^2}) dv.
\end{align}
\end{prop}

\subsection{Mass transplantation}$\ $

Finally, we state the mass transplantation argument due to Weinberger \cite{Wei56}, which is a powerful tool for estimates involving integrals. A clear proof can be seen in \cite{FL21}.
\begin{thm}[Mass transplantation]\label{masstransplantation}
Let $\Omega$ be a bounded Lipschitz domain in $\mathbb{H}^n$ with the same volume as $B(R)$. If $f(r)$ is decreasing and integrable on $[0,+\infty)$, then
\begin{equation}\label{mass-ineq}
    \int_{\Omega}f(r) dv\leq\int_{B(R)}f(r) dv.
\end{equation}
If in addition $f(r)$ is strictly decreasing, then equality holds if and only if $\Omega=B(R)$. If $f(r)$ is increasing and integrable on $[0,+\infty)$, then the inequality reverses direction.
\begin{proof}
Since $f$ is decreasing and $|\Omega|=|B(R)|$, we have
\begin{equation}\label{mass-ineq2}
\begin{split}
     \int_{\Omega}f(r) dv
    =&\int_{\Omega\cap B(R)}f(r) dv+\int_{\Omega\backslash B(R)}f(r) dv\\
    \leq& \int_{\Omega\cap B(R)}f(r) dv+|\Omega\backslash B(R)|f(R)\\
    =&\int_{\Omega\cap B(R)}f(r) dv+|B(R)\backslash \Omega|f(R)\\
    \leq& \int_{\Omega\cap B(R)}f(r) dv+\int_{B(R)\backslash \Omega}f(r) dv=\int_{B(R)}f(r) dv.
\end{split}
\end{equation}
It suffices to show that the inequality (\ref{mass-ineq}) is strict when $f(r)$ is strictly decreasing. If $\Omega\nsubseteq B(R)$, then $\Omega$ contains a point at radius $r>R$ and thus contains a neighborhood outside $B(R)$. Consequently, $|\Omega\backslash B(R)|>0$, and the first inequality in (\ref{mass-ineq2}) is strict because $f(r)$ is strictly decreasing. Similarly, if $B(R)\nsubseteq \Omega$, the second inequality in (\ref{mass-ineq2}) is strict. This completes the proof of Theorem \ref{masstransplantation}.
\end{proof}
\end{thm}

\section{The volume-area ratio function $g(r)$}\label{sec-3}
\subsection{The geometric properties of $g(r)$}$\ $

By using the expression for $g(r)$ given in \eqref{eq-def-g}, we have
\begin{align}\label{eq-deri-ln-1g}
(\lambda^{n-1}(r)g(r))'=(\int_0^r\lambda^{n-1}(t)dt)'=\lambda^{n-1}(r),
\end{align}
and then
\begin{align}\label{eq-deri-g}
g'=1-(n-1)\frac{\lambda' g}{\lambda}.
\end{align}

\begin{prop}\label{prop-divgY}
Let $Y:= g(r) \partial_r$ be a vector field. Then
\begin{equation}\label{div-Y}
    \divg(Y)=1.
\end{equation}

\begin{proof}
For each $p\in \mathbb{H}^n$, choose an orthonormal frame $\{e_1,\ldots,e_n\}$ around $p$ such that $e_n=\partial_r$. Recall the vector field $V:=\lambda(r)\partial_r$ is a conformal Killing field on $\mathbb{H}^n$, which satisfies
\begin{align}\label{eq-killing}
\la \nabla_{e_i}V,e_j\ra=\lambda'(r)\la e_i,e_j\ra,
\end{align}
see e.g. \cite[Lemma 2.2]{Bre13}. Then it follows from \eqref{eq-deri-g} that
\begin{align*}
    \divg(Y)=&\divg(\frac{g}{\lambda}V)
    =\sum_{i=1}^n\la \nabla_{e_i}(\frac{g}{\lambda}V),e_i\ra\\
    =&(\frac{g}{\lambda})'\lambda+n(\frac{g}{\lambda})\lambda'
    =g'+(n-1)\frac{\lambda' g}{\lambda}=1.
\end{align*}

\end{proof}
\end{prop}

A direct corollary follows from Proposition \ref{prop-divgY} and the divergence theorem.
\begin{cor}\label{cor-g}
We have the following inequality
\begin{align}\label{eq-g}
\int_{\Sigma}g d\mu\geq |\Omega|.
\end{align}
\begin{proof}
Since $\la \partial_r,\nu\ra\leq 1$ and $g\geq 0$, we have
\begin{align*}
    |\Omega|=\int_{\Omega}dv=\int_{\Omega}\divg(g(r)\partial_ r) dv=\int_{\Sigma}g\la \partial_r,\nu\ra d\mu\leq \int_{\Sigma}g d\mu.
\end{align*}
\end{proof}
\end{cor}


Similarly, the following Proposition follows just from the divergence theorem and the mass transplantation, as shown in \cite[Lemma 3.3]{BS14}. For the sake of completeness, let us give a proof.
\begin{prop}\label{prop-g2-n=4}
Let $B(R)$ be a geodesic ball in $\mathbb{H}^n$ with the same volume as $\Omega$, then
\begin{align}\label{eq-intg2}
\int_{\Sigma}g^2 d\mu\geq g(R)|B(R)|.
\end{align}
\begin{proof}
First, it follows from \eqref{eq-deri-g} and \eqref{eq-killing} that
\begin{align*}
    \divg(g^2\partial_r)=\divg(\frac{g^2}{\lambda}V)=(\frac{g^2}{\lambda})'\lambda+n(\frac{g^2}{\lambda})\lambda'=g(1+g'),
\end{align*}
and then
\begin{align*}
    \int_{\Sigma}g^2 d\mu\geq \int_{\Sigma}g^2\la \partial_r,\nu\ra d\mu=\int_{\Omega}g(1+g') dv.
\end{align*}

Next, we demonstrate the monotonicity of $g(1+g')$ as follows:
\begin{align*}
    (g(1+g'))'=& g'+(g')^2+gg''\\
    =& g'+(g')^2+g((n-1)\frac{g}{\lambda^2}-(n-1)\frac{\lambda'g'}{\lambda})\\
    =& g'(1-(n-1)\frac{\lambda'g}{\lambda})+(g')^2+(n-1)\frac{g^2}{\lambda^2}\\
    =& 2(g')^2+(n-1)\frac{g^2}{\lambda^2}\geq 0.
\end{align*}
Using the mass transplantation argument, we conclude
\begin{align*}
    \int_{\Omega}g(1+g') dv & \geq \int_{B(R)}g(1+g') dv
    = \omega_{n-1}\int_0^R g(1+g')\lambda^{n-1}dt\\
    &= \omega_{n-1}\int_0^R (\lambda^{n-1}g^2)' dt
    = \omega_{n-1} \lambda^{n-1}(R)g^2(R)\\
    &= g(R) |B(R)|.
\end{align*}
This complete the proof of (\ref{eq-intg2}).
\end{proof}
\end{prop}

\subsection{The analytic properties of $g(r)$}$\ $

The following properties are essential in our proof.

\begin{prop}\label{prop-g}
The function $g$ is an increasing concave function, with the following limits:
\begin{align}\label{eq-g-lim}
\lim_{r\to 0}\frac{g}{\lambda}=\frac{1}{n},\qquad \lim_{r\to \infty}g=\frac{1}{n-1},\qquad \lim_{r\to 0}g'=\frac{1}{n},\qquad \lim_{r\to \infty}g'=0.
\end{align}
Furthermore, if $n\geq 4$, we have
\begin{align}\label{eq-g-lim2}
\lim_{r\to\infty}\lambda^2g'=\frac{1}{n-3}.
\end{align}
\begin{proof}
First, since $\left.\lambda^{n-1}(g-\frac{1}{n-1}\frac{\lambda}{\lambda'})\right|_{r=0}=0$ and
\begin{align*}
    (\lambda^{n-1}(g-\frac{1}{n-1}\frac{\lambda}{\lambda'}))'=\lambda^{n-1}-\frac{1}{n-1}\frac{n\lambda^{n-1}(\lambda')^2-\lambda^{n+1}}{(\lambda')^2}\leq 0,
\end{align*}
we deduce
\begin{align*}
    g\leq \frac{1}{n-1}\frac{\lambda}{\lambda'},
\end{align*}
and then
\begin{align*}
    g'=1-(n-1)\frac{\lambda' g}{\lambda}\geq 0.
\end{align*}

Furthermore, since $\left.\lambda^{n-1}(g-\frac{\lambda\lambda'}{(n-1)\lambda^2+n})\right|_{r=0}=0$ and
\begin{align*}
    (\lambda^{n-1}(g-\frac{\lambda\lambda'}{(n-1)\lambda^2+n}))'=\lambda^{n-1}-\lambda^{n-1}\frac{(n-1)^2\lambda^4+(2n^2-2n+1)\lambda^2+n^2}{((n-1)\lambda^2+n)^2}\leq 0,
\end{align*}
we have
\begin{align*}\label{eq-gll'}
g\leq \frac{\lambda\lambda'}{(n-1)\lambda^2+n},
\end{align*}
which implies
\begin{align*}
    g''=(n-1)\frac{g}{\lambda^2}-(n-1)g'\frac{\lambda'}{\lambda}=\frac{n-1}{\lambda^2}(g((n-1)\lambda^2+n)-\lambda\lambda')\leq 0.
\end{align*}

Finally, the limits follow from l'H\^opitcal rule:
\begin{align*}
    \lim_{r\to 0}\frac{g}{\lambda}
    =& \lim_{r\to 0}\frac{\int_0^r\lambda^{n-1}(t)dt}{\lambda^n(r)}=\lim_{r\to 0}\frac{\lambda^{n-1}(r)}{n\lambda^{n-1}(r)\lambda'(r)}=\frac{1}{n},\\
    \lim_{r\to \infty}g
    =& \lim_{r\to \infty}\frac{\int_0^r\lambda^{n-1}(t)dt}{\lambda^{n-1}(r)}=\lim_{r\to \infty}\frac{\lambda^{n-1}(r)}{(n-1)\lambda^{n-2}(r)\lambda'(r)}=\frac{1}{n-1},\\
    \lim_{r\to 0}g'
    =& 1-(n-1)\lim_{r\to 0}\frac{\lambda'g}{\lambda}=1-(n-1)\lim_{r\to 0}\frac{\int_0^r\lambda^{n-1}(t)dt}{\lambda^n(r)}=1-\frac{n-1}{n}=\frac{1}{n},\\
    \lim_{r\to \infty}g'
    =& 1-(n-1)\lim_{r\to \infty}\frac{\lambda'g}{\lambda}=1-(n-1)\cdot\frac{1}{n-1}=0,
\end{align*}
and if $n\geq 4$,
\begin{align*}
    \lim_{r\to \infty}\lambda^2g'
    =& \lim_{r\to \infty}\frac{\lambda^n-(n-1)\lambda'\int_0^r\lambda^{n-1}}{\lambda^{n-2}}=\lim_{r\to\infty}\frac{\frac{\lambda^n}{\lambda'}-(n-1)\int_0^r\lambda^{n-1}}{\lambda^{n-3}}\\
    =& \lim_{r\to\infty}\frac{\frac{(n-1)\lambda^{n+1}+n\lambda^{n-1}}{\lambda^2+1}-(n-1)\lambda^{n-1}}{(n-3)\lambda^{n-4}\lambda'}=\lim_{r\to\infty}\frac{\lambda^{n-1}}{(n-3)\lambda^{n-4}\lambda'(\lambda^2+1)}=\frac{1}{n-3}.
\end{align*}

\end{proof}
\end{prop}

Since $\frac{\lambda' g}{\lambda}=\frac{1}{n-1}(1-g')$, Proposition \ref{prop-g} immediately yields the following proposition.
\begin{prop}\label{prop-l'g/l}
The function
\begin{align*}
    \frac{\lambda'(r)g(r)}{\lambda(r)}
\end{align*}
is increasing, and
\begin{align*}
    \frac{1}{n}\leq \frac{\lambda'(r)g(r)}{\lambda(r)}\leq \frac{1}{n-1}.
\end{align*}
\end{prop}

On the other hand, we have
\begin{prop}\label{prop-ll'g'/g}
The function
\begin{align*}
    \frac{\lambda(r)\lambda'(r)g'(r)}{g(r)}
\end{align*}
is increasing, and if $n\geq 4$, we have
\begin{align*}
    \frac{\lambda(r)\lambda'(r)g'(r)}{g(r)}\leq \frac{n-1}{n-3}.
\end{align*}

\begin{proof}
First, we show that $\frac{\lambda\lambda'g'}{g}$ is increasing. Note that
\begin{align*}
    (\frac{\lambda\lambda'g'}{g})'=(\frac{\lambda\lambda'}{g}-(n-1)(\lambda^2+1))'=\frac{((n+1)\lambda^2+n)g-\lambda\lambda'}{g^2}-2(n-1)\lambda\lambda'.
\end{align*}
It suffices to show that
\begin{align*}
    ((n+1)\lambda^2+n)g-\lambda\lambda'\geq 2(n-1)\lambda\lambda'g^2,
\end{align*}
which is equivalent to
\[
(g-\frac{(n+1)\lambda^2+n}{4(n-1)\lambda\lambda'})^2\leq \frac{(n-3)^2\lambda^4+2(n^2-3n+4)\lambda^2+n^2}{(4(n-1)\lambda\lambda')^2}.
\]

For $n=2$, direct calculation shows that $\frac{\lambda\lambda'g'}{g}=\lambda'$ is increasing. From now on, we focus on the case $n\geq 3$. By Proposition \ref{prop-g}, we have
\begin{align*}
    g\leq \frac{\lambda}{(n-1)\lambda'}\leq \frac{(n+1)\lambda^2+n}{4(n-1)\lambda\lambda'}.
\end{align*}
Thus we need only to prove that
\begin{align*}
    g\geq \frac{(n+1)\lambda^2+n-\sqrt{(n-3)^2\lambda^4+2(n^2-3n+4)\lambda^2+n^2}}{4(n-1)\lambda\lambda'}.
\end{align*}

For convenience, denote
\begin{align*}
    A=\sqrt{(n-3)^2\lambda^4+2(n^2-3n+4)\lambda^2+n^2},
\end{align*}
then $A'=\frac{2(n-3)^2\lambda^3\lambda'+2(n^2-3n+4)\lambda\lambda'}{A}$. It suffices to show that
\begin{align*}
    \lambda^{n-1}g\geq \frac{2\lambda^n\lambda'}{(n+1)\lambda^2+n+A}.
\end{align*}
Equality holds at $r=0$. By taking a derivative, we need only to prove that 
\begin{align*}
    & \lambda^{n-1}\geq \frac{2((n+1)\lambda^{n+1}+n\lambda^{n-1})((n+1)\lambda^2+n+A)-2\lambda^n\lambda'(2(n+1)\lambda\lambda'+A')}{((n+1)\lambda^2+n+A)^2} \\
    \Longleftrightarrow & ((n+1)\lambda^2+n+A)^2\geq 2((n+1)\lambda^2+n)((n+1)\lambda^2+n+A)-2\lambda\lambda'(2(n+1)\lambda\lambda'+A') \\
    \Longleftrightarrow & 2\lambda\lambda'(2(n+1)\lambda\lambda'+\frac{2(n-3)^2\lambda^3\lambda'+2(n^2-3n+4)\lambda\lambda'}{A})\geq ((n+1)\lambda^2+n)^2-A^2\\
    \Longleftrightarrow & 2\lambda(2(n+1)\lambda+\frac{2(n-3)^2\lambda^3+2(n^2-3n+4)\lambda}{A})(\lambda^2+1)\geq 8(n-1)\lambda^2(\lambda^2+1)\\
     \Longleftrightarrow & (n-3)^2\lambda^2+(n^2-3n+4)\geq (n-3)A.
\end{align*}
By squaring both sides of the above inequality, it suffices to show that
\begin{align*}
         & (n-3)^4\lambda^4+2(n-3)^2(n^2-3n+4)\lambda^2+(n^2-3n+4)^2 \\
    \geq &  (n-3)^2((n-3)^2\lambda^4+2(n^2-3n+4)\lambda^2+n^2),
\end{align*}
which is clearly correct. Therefore, the function $\frac{\lambda\lambda'g'}{g}$ is increasing. 

Moreover, it follows from \eqref{eq-g-lim} and \eqref{eq-g-lim2} that
\[
\frac{\lambda\lambda'g'}{g}\leq \lim_{r\to\infty}\frac{\lambda\lambda'g'}{g}=\lim_{r\to\infty}\frac{\lambda^2g'}{g}=\frac{n-1}{n-3}.
\]
\end{proof}
\end{prop}

By Proposition \ref{prop-ll'g'/g}, the function $\frac{\lambda^2g'}{g}=\frac{\lambda\lambda'g'}{g}\cdot\frac{\lambda}{\lambda'}$ is a product of two positive increasing functions. Hence we have
\begin{prop}\label{prop-l2g'/g-n=4}
The function
\begin{align*}
    \frac{\lambda^2(r) g'(r)}{g(r)}
\end{align*}
is increasing.
\end{prop}

Notice that 
\begin{align}
    \divg(gg'\partial_r)=(g')^2+(n-1)\frac{g^2}{\lambda^2}.
\end{align}
Then the expression \eqref{eq-sigmabr} for $\sigma_1(B(r))$ can be rewritten as
\begin{equation}\label{eq-sigmabr2}
    \sigma_1(B(r))=\frac{\int_{S(r)}gg'd\mu}{g(r)^2|S(r)|}=\frac{g'(r)}{g(r)}.
\end{equation}
Combining Proposition \ref{prop-l2g'/g-n=4} with (\ref{eq-S}) and (\ref{eq-sigmabr2}), we obtain the following key lemma.
\begin{lem}\label{lem-ln-1g'/g}
The quantities
\begin{align}\label{eq-mono-n=4}
|S(r)|^{\frac{2}{n-1}}\sigma_1(B(r))
\end{align}
and
\begin{align}\label{eq-mono-n>=5}
|S(r)|\sigma_1(B(r))
\end{align}
are both increasing functions of the radius $r$.
\end{lem}

\section{Weinstock inequality in Hyperbolic space}\label{sec-4}

From now on, let $R$ be the positive number such that $|B(R)|=|\Omega|$ in $\mathbb{H}^n$.
\subsection{The auxiliary function $h$} $\ $

Since $\int_{B(t)}\frac{\lambda'g}{\lambda}dv=\omega_{n-1}\int_0^t \lambda^{n-2}\lambda'g(s) ds$ is an integral that increases from zero to infinity, we can define an auxiliary function $h:(0,+\infty)\to (0,+\infty)$ as follows: 
\begin{align}\label{def-h}
h(\int_{B(t)}\frac{\lambda'g}{\lambda}dv)=\frac{1}{|S(t)|}.
\end{align}

The following properties hold for the function $h$.
\begin{lem}
The function $h$ is decreasing and log-convex, such that
\begin{align}\label{relationofh'h}
\frac{h'(\int_{B(t)}\frac{\lambda'g}{\lambda}dv)}{h(\int_{B(t)}\frac{\lambda'g}{\lambda}dv)}=-\frac{n-1}{|B(t)|}.
\end{align}
\begin{proof}
Since $\int_{B(t)}\frac{\lambda'g}{\lambda}dv$ is an increasing function of $t$, while $\frac{1}{|S(t)|}$ is a decreasing function of $t$, we can conclude that $h$ is a decreasing function. 

Next, differentiating both sides of \eqref{def-h} yields
\begin{align*}
    h'(\int_{B(t)}\frac{\lambda'g}{\lambda}dv)|S(t)|\frac{\lambda'(t)g(t)}{\lambda(t)}=-\frac{(n-1)\lambda'(t)}{|S(t)|\lambda(t)}.
\end{align*}
Combining this with $g(t)=\frac{|B(t)|}{|S(t)|}$, we obtain \eqref{relationofh'h}.

Finally, the log-convexity of $h$ follows from the fact that $-\frac{n-1}{|B(t)|}$ is an increasing function of $t$, and $\int_{B(t)}\frac{\lambda'g}{\lambda}dv$ is also an increasing function of $t$.

\end{proof}
\end{lem}

\subsection{An inequality from the IMCF}$\ $

Suppose $\Omega$ be a smooth bounded domain in $\mathbb{H}^n$ with star-shaped mean convex boundary $\Sigma=\partial\Omega$. Then $\Sigma$ can be written as a graph over $\mathbb{S}^{n-1}$, i.e. there exists a smooth embedding $X_0:\mathbb{S}^{n-1}\to\mathbb{H}^{n}$ with $X_0(\mathbb{S}^{n-1})=\Sigma$. 

We evolve $\Sigma$ along the IMCF $X:\mathbb{S}^{n-1}\times[0,\infty)\to\mathbb{H}^n$ solving the equation
\begin{align}
    \left\{\begin{array}{l}
    \frac{\partial}{\partial t}X=\frac{1}{H}\nu, \\ 
    X(\cdot,0)=X_0(\cdot),
    \end{array}\right.
\end{align}
where $\nu$ is the outward unit normal vector field of the flow hypersurface $\Sigma_t=X(\mathbb{S}^{n-1},t)$ and $H$ is the mean curvature of $\Sigma_t$.

Now we consider the following monotone quantity along IMCF.
\begin{prop}
Along IMCF, the quantity
\[
A(t):=|\Sigma_t|^{-1}\frac{\int_{\Sigma_t}g d\mu_t}{|\Omega_t|h(\int_{\Omega_t}\frac{\lambda'g}{\lambda} dv_t)}
\]
is decreasing, i.e. $A'(t)\leq 0$, and $A'(t)=0$ if and only if $\Sigma$ is a geodesic sphere centered at the origin.
\begin{proof}
Firstly, along IMCF, the area element $d\mu_t$ evolves by
\[
\frac{\partial}{\partial t}d\mu_t=d\mu_t,
\]
and the co-area formula implies
\[
\frac{d}{dt}\int_{\Omega_t}\frac{\lambda'g}{\lambda}dv_t=\int_{\Sigma_t}\frac{\lambda'g}{\lambda H}d\mu_t.
\]

Assume that $R_t$ satisfies $|\Omega_t|=|B(R_t)|$. For simplicity, we omit the subscript $t$ for $R_t$, $\Omega_t$ and $\Sigma_t$ and omit the area element $d\mu_t$ and volume element $dv_t$ in each integrations. Then along ICMF, we have
\begin{align*}
    \frac{d}{dt}\frac{\int_{\Sigma}g}{|\Omega|h(\int_{\Omega}\frac{\lambda'g}{\lambda})}
    =&\frac{\int_{\Sigma}g+\int_{\Sigma}\frac{g'\la \partial_r,\nu\ra}{H}}{|\Omega|h(\int_{\Omega}\frac{\lambda'g}{\lambda})}-\frac{\int_{\Sigma}g(h(\int_{\Omega}\frac{\lambda'g}{\lambda})\int_{\Sigma}\frac{1}{H}+|\Omega|h'(\int_{\Omega}\frac{\lambda'g}{\lambda})\int_{\Sigma}\frac{\lambda'g}{\lambda H})}{|\Omega|^2h(\int_{\Omega}\frac{\lambda'g}{\lambda})^2}\\
    \leq&\frac{\int_{\Sigma}g}{|\Omega|h(\int_{\Omega}\frac{\lambda'g}{\lambda})}+\frac{\int_{\Sigma}\frac{g'}{H}}{|\Omega|h(\int_{\Omega}\frac{\lambda'g}{\lambda})}-\frac{\int_{\Sigma}g(\int_{\Sigma}\frac{1}{H}+|\Omega|\frac{h'(\int_{\Omega}\frac{\lambda'g}{\lambda})}{h(\int_{\Omega}\frac{\lambda'g}{\lambda})}\int_{\Sigma}\frac{\lambda'g}{\lambda H})}{|\Omega|^2h(\int_{\Omega}\frac{\lambda'g}{\lambda})}.
\end{align*}

Since $\frac{\lambda'g}{\lambda}$ is increasing by Proposition \ref{prop-l'g/l}, the mass transplantation Theorem \ref{masstransplantation} yields
\begin{align}\label{eq-int-l'g/l}
\int_{\Omega}\frac{\lambda'g}{\lambda} dv\geq \int_{B(R)}\frac{\lambda'g}{\lambda} dv.
\end{align}
Furthermore, combining this with the log-convexity of $h$ and \eqref{relationofh'h}, we obtain
\begin{align*}
    \frac{h'(\int_{\Omega}\frac{\lambda'g}{\lambda})}{h(\int_{\Omega}\frac{\lambda'g}{\lambda})}\geq \frac{h'(\int_{B(R)}\frac{\lambda'g}{\lambda})}{h(\int_{B(R)}\frac{\lambda'g}{\lambda})}=-\frac{n-1}{|B(R)|}=-\frac{n-1}{|\Omega|}.
\end{align*}
Thus it follows from \eqref{eq-deri-g} and \eqref{eq-g} that
\begin{align*}
    \frac{d}{dt}\frac{\int_{\Sigma}g}{|\Omega|h(\int_{\Omega}\frac{\lambda'g}{\lambda})}
    \leq & \frac{\int_{\Sigma}g}{|\Omega|h(\int_{\Omega}\frac{\lambda'g}{\lambda})}+\frac{\int_{\Sigma}\frac{g'}{H}}{|\Omega|h(\int_{\Omega}\frac{\lambda'g}{\lambda})}-\frac{\int_{\Sigma}g\int_{\Sigma}\frac{g'}{H}}{|\Omega|^2h(\int_{\Omega}\frac{\lambda'g}{\lambda})}\\
    \leq & \frac{\int_{\Sigma}g}{|\Omega|h(\int_{\Omega}\frac{\lambda'g}{\lambda})},
\end{align*}
which leads the decreasing property of $A(t)$. Equality holds if and only if $\la \partial_r,\nu\ra=1$ everywhere on $\Sigma_t$, which is equivalent to that $\Sigma_t$ is a geodesic sphere centered at the origin.
\end{proof}
\end{prop}

As a corollary, combining with H\"older inequality, we have
\begin{cor}\label{cor-int-g}
\begin{align}\label{eq-intg2-imcf}
\int_{\Sigma}g^2 d\mu\geq |\Sigma||\Omega|^2h(\int_{\Omega}\frac{\lambda'g}{\lambda}dv)^2.
\end{align}
\begin{proof}
By the result of Gerhardt \cite{Ger11}, the flow hypersurface $\Sigma_t$ of IMCF remains to be star-shaped, mean convex, expands to infinity and the principal curvatures $\kappa_i$ decays to $1$ exponentially as $t\to \infty$. Let $g^{\Sigma_t}_{ij}$ be the induced metric on $\Sigma_t$. The asymptotical behavior of $\Sigma_t$ along IMCF proved by Gerhardt \cite{Ger11} implies that
\[\begin{aligned}
\sqrt{\det g^{\Sigma_t}}=&\lambda^n\sqrt{\det g_{\bS^{n-1}}}\left(1+O(e^{-\frac{2t}{n-1}})\right),\\
\lambda(r_t)=&O(e^{\frac{t}{n-1}}).
\end{aligned}\]
Furthermore, we have
\[
\lim_{t\to\infty}r_t=\lim_{t\to\infty}R_t=+\infty.
\]
Then it follows from the monotonicity of $h$, \eqref{def-h} and \eqref{eq-int-l'g/l} that
\[\begin{aligned}
A(t)=&|\Sigma_t|^{-1}\frac{\int_{\Sigma_t}g d\mu_t}{|\Omega_t|h(\int_{\Omega_t}\frac{\lambda'g}{\lambda} dv_t)}\geq|\Sigma_t|^{-1}\frac{\int_{\Sigma_t}g d\mu_t}{|\Omega_t|h(\int_{B(R_t)}\frac{\lambda'g}{\lambda} dv_t)}\\
=&\frac{|S(R_t)|\int_{\Sigma_t}g d\mu_t}{|\Sigma_t||B(R_t)|}=\frac{1}{g(R_t)}\cdot \frac{\int_{\Sigma_t}gd\mu_t}{|\Sigma_t|}.
\end{aligned}\]
Using Mean Value Theorem, there exists $\xi_t\in [\min\limits_{\Sigma_t}r_t,\max\limits_{\Sigma_t}r_t]$ such that
\[
\frac{\int_{\Sigma_t}gd\mu_t}{|\Sigma_t|}=g(\xi_t).
\]
Since $A(t)$ is decreasing, we obtain from \eqref{eq-g-lim} that
\begin{align*}
    \frac{\int_{\Sigma}g d\mu}{|\Omega||\Sigma|h(\int_{\Omega}\frac{\lambda'g}{\lambda} dv)}=A(0)\geq \liminf_{t\to\infty}A(t)\geq \liminf_{t\to\infty}\frac{g(\xi_t)}{g(R_t)}=\frac{n-1}{n-1}=1.
\end{align*}
Combining with Cauchy-Schwarz inequality, we conclude
\begin{align*}
    \int_{\Sigma}g^2 d\mu \geq\frac{1}{|\Sigma|}(\int_{\Sigma}g d\mu)^2\geq |\Sigma||\Omega|^2h(\int_{\Omega}\frac{\lambda'g}{\lambda} dv)^2.
\end{align*}

\end{proof}
\end{cor}

\subsection{Proof of Theorem \ref{main thm} for $n\geq 5$}\label{subsection4.3}$\ $

Before beginning the proof, we need the following lemma:
\begin{lem}\label{lem-n=5}
If $n\geq 5$, then the function
\[
(g')^2+(n-1)\frac{g^2}{\lambda^2}+\frac{2(n-1)}{n}\frac{\lambda'g}{\lambda}
\]
is decreasing in $r$. 
\begin{proof}
In fact, by taking a derivative, we have
\[\begin{aligned}
&((g')^2+(n-1)\frac{g^2}{\lambda^2}+\frac{2(n-1)}{n}\frac{\lambda'g}{\lambda})'=((g')^2+(n-1)\frac{g^2}{\lambda^2}+\frac{2}{n}(1-g'))'\\
=&2g'g''+2(n-1)\frac{g}{\lambda}\frac{g'\lambda-\lambda'g}{\lambda^2}-\frac{2}{n}g''=2g''(g'-\frac{1}{n})+2(n-1)\frac{g}{\lambda}\frac{g'-\frac{\lambda'g}{\lambda}}{\lambda}\\
=&2(n-1)(\frac{g}{\lambda^2}-\frac{\lambda'g'}{\lambda})(g'-\frac{1}{n})+2n\frac{g}{\lambda}\frac{g'-\frac{1}{n}}{\lambda}=2(n-1)(g'-\frac{1}{n})\frac{g}{\lambda^2}(\frac{2n-1}{n-1}-\frac{\lambda\lambda'g'}{g}).
\end{aligned}\]
Thus it suffices to show that
\begin{align}\label{eq-deccondi}
\frac{\lambda\lambda'g'}{g}\leq \frac{2n-1}{n-1}.
\end{align}
Fortunately, by use of Proposition \ref{prop-ll'g'/g}, we know that \eqref{eq-deccondi} holds for $n\geq 5$. 
\end{proof}
\end{lem}

From this lemma, by the mass transplantation, we have
\begin{align*}
    \int_{\Omega}((g')^2+(n-1)\frac{g^2}{\lambda^2}+\frac{2(n-1)}{n}\frac{\lambda'g}{\lambda})\leq \int_{B(R)}((g')^2+(n-1)\frac{g^2}{\lambda^2}+\frac{2(n-1)}{n}\frac{\lambda'g}{\lambda}),
\end{align*}
then
\begin{align}\label{eq-afterdeccondi}
\frac{2(n-1)}{n}(\int_{\Omega}\frac{\lambda'g}{\lambda}-\int_{B(R)}\frac{\lambda'g}{\lambda})\leq g'(R)|B(R)|-\int_{\Omega}((g')^2+(n-1)\frac{g^2}{\lambda^2}).
\end{align}

We proceed to prove Theorem \ref{main thm} for $n\geq 5$. Firstly, it follows from the upper bound \eqref{eq-upperbound-steklov} and Corollary \ref{cor-int-g} that
\begin{align}\label{eq-uppbdimcf}
\sigma_1(\Omega)\leq\frac{\int_{\Omega}((g')^2+(n-1)\frac{g^2}{\lambda^2} ) dv}{|\Omega|^2|\Sigma|h(\int_{\Omega}\frac{\lambda'g}{\lambda} dv)^2}.
\end{align}
The following lemma is essential, which is slightly stronger than Theorem \ref{main thm}.

\begin{lem}\label{main lem}
Let $n\geq 5$. If $\Omega\subset\mathbb{H}^n$ is a smooth domain with star-shaped mean convex boundary $\Sigma=\partial\Omega$, and $B(R)$ is a geodesic ball with the same volume as $\Omega$. Then
\begin{align}\label{eq-prelem}
|\Sigma|\sigma_1(\Omega)\leq |S(R)|\sigma_1(B(R)).
\end{align}
\begin{proof}
We estimate $h(\int_{\Omega}\frac{\lambda'g}{\lambda})^2$ in the denominator of the right hand side of \eqref{eq-uppbdimcf}. Using Mean Value Theorem, there exists $\xi\in [\int_{B(R)}\frac{\lambda'g}{\lambda},\int_{\Omega}\frac{\lambda'g}{\lambda}]$ such that
\[\begin{aligned}
h(\int_{\Omega}\frac{\lambda'g}{\lambda})^2-h(\int_{B(R)}\frac{\lambda'g}{\lambda})^2=2h(\xi)h'(\xi)(\int_{\Omega}\frac{\lambda'g}{\lambda}-\int_{B(R)}\frac{\lambda'g}{\lambda}).
\end{aligned}\]
Recall that the function
\begin{align*}
    h(\int_{B(r)}\frac{\lambda'g}{\lambda})h'(\int_{B(r)}\frac{\lambda'g}{\lambda})=-\frac{n-1}{|B(r)||S(r)|^2}
\end{align*}
is increasing in $r$. Combining this with $g'(R)\leq \frac{1}{n}$ and \eqref{eq-afterdeccondi}, we have
\begin{align*}
    h(\int_{\Omega}\frac{\lambda'g}{\lambda})^2
    \geq & h(\int_{B(R)}\frac{\lambda'g}{\lambda})^2+2h(\int_{B(R)}\frac{\lambda'g}{\lambda})h'(\int_{B(R)}\frac{\lambda'g}{\lambda})(\int_{\Omega}\frac{\lambda'g}{\lambda}-\int_{B(R)}\frac{\lambda'g}{\lambda})\\
    = & \frac{1}{|S(R)|^2}-\frac{2(n-1)}{|B(R)||S(R)|^2}(\int_{\Omega}\frac{\lambda'g}{\lambda}-\int_{B(R)}\frac{\lambda'g}{\lambda})\\
    \geq & \frac{1}{|S(R)|^2}-\frac{1}{g'(R)|B(R)||S(R)|^2}\frac{2(n-1)}{n}(\int_{\Omega}\frac{\lambda'g}{\lambda}-\int_{B(R)}\frac{\lambda'g}{\lambda})\\
    \geq & \frac{1}{|S(R)|^2}-\frac{1}{g'(R)|B(R)||S(R)|^2}(g'(R)|B(R)|-\int_{\Omega}((g')^2+(n-1)\frac{g^2}{\lambda^2}))\\
    = & \frac{1}{g'(R)|B(R)||S(R)|^2}\int_{\Omega}((g')^2+(n-1)\frac{g^2}{\lambda^2}).
\end{align*}
Therefore, by (\ref{eq-uppbdimcf}) and (\ref{eq-sigmabr2}), we derive
\begin{align*}
    |\Sigma|\sigma_1(\Omega)
    \leq \frac{\int_{\Omega}(g')^2+(n-1)\frac{g^2}{\lambda^2}}{|\Omega|^2h(\int_{\Omega}\frac{\lambda'g}{\lambda})^2}
    \leq \frac{g'(R)|B(R)||S(R)|^2}{|B(R)|^2}=|S(R)|\sigma_1(B(R)).
\end{align*}
This completes the proof of \eqref{eq-prelem}.
\end{proof}
\end{lem}

\begin{proof}[Proof of Theorem \ref{main thm} for $n\geq 5$]
The isoperimetric inequality in Theorem \ref{thm-isoperimetric inequality} implies that the radius of $\Omega^{\ast}$ is no less than $R$. By Lemma \ref{main lem} and \eqref{eq-mono-n>=5}, we obtain
\begin{align*}
    \sigma_1(\Omega)\leq \frac{|S(R)|\sigma_1(B(R))}{|\Sigma|}
    \leq \frac{|\partial\Omega^{\ast}|\sigma_1(\Omega^{\ast})}{|\Sigma|}=\sigma_1(\Omega^{\ast}).
\end{align*}
The equality holds if and only if the radius of $\Omega^{\ast}$ is $R$. That is, $\Omega$ is a geodesic ball, due to the rigidity of the isoperimetric inequality in Theorem \ref{thm-isoperimetric inequality}.
\end{proof}

\subsection{Proof of Theorem \ref{main thm} for $n=4$}$\ $

In this subsection, we consider the case $n=4$. In this case, we can explicitly compute $g$ and $g'$ as
\begin{align}\label{eq-g-n=4}
g(r)=\frac{1}{\lambda^3(r)}\int_0^r\lambda^3(t)dt=\frac{\lambda(\lambda'+2)}{3(\lambda'+1)^2},\quad 
g'(r)=\frac{1}{(\lambda'+1)^2}.
\end{align}

Firstly, we use the following lemma instead of Lemma \ref{lem-n=5}.
\begin{lem}\label{lem-n=4}
If $n=4$, then the function
\begin{align*}
    (g')^2+\frac{3 g^2}{\lambda^2}+\frac{\lambda'g}{\lambda}
\end{align*}
is decreasing in $r$.
\begin{proof}
Applying \eqref{eq-g-n=4} into the formula, we have
\[
(g')^2+\frac{3 g^2}{\lambda^2}+\frac{\lambda'g}{\lambda}=\frac{2(\lambda')^2+6\lambda'+7}{3(\lambda'+1)^4}=\frac{2}{3(\lambda'+1)^2}+\frac{2}{3(\lambda'+1)^3}+\frac{1}{(\lambda'+1)^4},
\]
which is obviously decreasing.
\end{proof}
\end{lem}

It follows from Lemma \ref{lem-n=4} and the mass transplantation argument that
\begin{align}\label{eq-afterdeccondi-n=4}
\int_{\Omega}\frac{\lambda'g}{\lambda}-\int_{B(R)}\frac{\lambda'g}{\lambda}\leq g'(R)|B(R)|-\int_{\Omega}((g')^2+\frac{3 g^2}{\lambda^2}).
\end{align}
Similar to the proof of Lemma \ref{main lem}, but slightly different because we use \eqref{eq-mono-n=4} instead of \eqref{eq-mono-n>=5}, we have

\begin{lem}
If $\Omega\subset\mathbb{H}^4$ is a smooth domain with star-shaped mean convex boundary $\Sigma=\partial\Omega$, and $B(R)$ is a geodesic ball with the same volume as $\Omega$. Then
\begin{align}\label{eq-prelem-n=4}
|\Sigma|^{\frac{2}{3}}\sigma_1(\Omega)\leq |S(R)|^{\frac{2}{3}}\sigma_1(B(R)).
\end{align}
\begin{proof}
From \eqref{eq-upperbound-steklov}, \eqref{eq-intg2-imcf} and \eqref{eq-intg2}, we have
\begin{align*}
    |\Sigma|^{\frac{2}{3}}\sigma_1(\Omega)
    & \leq \frac{\int_{\Omega}(g')^2+\frac{3 g^2}{\lambda^2}}{(\int_{\Sigma}g^2)^{\frac{1}{3}}(\frac{1}{|\Sigma|}\int_{\Sigma}g^2)^{\frac{2}{3}}}
    \leq \frac{\int_{\Omega}(g')^2+\frac{3 g^2}{\lambda^2}}{(\int_{\Sigma}g^2)^{\frac{1}{3}}(\frac{1}{|\Sigma|}\int_{\Sigma}g)^{\frac{4}{3}}}\\
    &\leq \frac{\int_{\Omega}(g')^2+\frac{3 g^2}{\lambda^2}}{(g(R)|B(R)|)^{\frac{1}{3}}(|\Omega|^2h(\int_{\Omega}\frac{\lambda'g}{\lambda})^2)^{\frac{2}{3}}}=\frac{\int_{\Omega}(g')^2+\frac{3 g^2}{\lambda^2}}{g(R)^{\frac{1}{3}}|B(R)|^{\frac{5}{3}}h(\int_{\Omega}\frac{\lambda'g}{\lambda})^{\frac{4}{3}}}.
\end{align*}

Next, again by Mean Value Theorem, there exists $\xi\in[\int_{B(R)}\frac{\lambda'g}{\lambda},\int_{\Omega}\frac{\lambda'g}{\lambda}]$ such that
\begin{align*}
    h(\int_{\Omega}\frac{\lambda'g}{\lambda})^{\frac{4}{3}}-h(\int_{B(R)}\frac{\lambda'g}{\lambda})^{\frac{4}{3}}
    =\frac{4}{3}h(\xi)^{\frac{1}{3}}h'(\xi)(\int_{\Omega}\frac{\lambda'g}{\lambda}-\int_{B(R)}\frac{\lambda'g}{\lambda}).
\end{align*}
We know that
\begin{align*}
    h(\int_{B(r)}\frac{\lambda'g}{\lambda})^{\frac{1}{3}}h'(\int_{B(r)}\frac{\lambda'g}{\lambda})
    =-\frac{3}{|B(r)||S(r)|^{\frac{4}{3}}}
\end{align*}
is increasing in $r$. Then it follows from $g'(R)\leq \frac{1}{4}$ and \eqref{eq-afterdeccondi-n=4} that
\begin{align*}
    h(\int_{\Omega}\frac{\lambda'g}{\lambda})^{\frac{4}{3}}
    \geq & h(\int_{B(R)}\frac{\lambda'g}{\lambda})^{\frac{4}{3}}+\frac{4}{3}h(\int_{B(R)}\frac{\lambda'g}{\lambda})^{\frac{1}{3}}h'(\int_{B(R)}\frac{\lambda'g}{\lambda})(\int_{\Omega}\frac{\lambda'g}{\lambda}-\int_{B(R)}\frac{\lambda'g}{\lambda})\\
    = & \frac{1}{|S(R)|^{\frac{4}{3}}}-\frac{4}{|B(R)||S(R)|^{\frac{4}{3}}}(\int_{\Omega}\frac{\lambda'g}{\lambda}-\int_{B(R)}\frac{\lambda'g}{\lambda})\\
    \geq & \frac{1}{|S(R)|^{\frac{4}{3}}}-\frac{1}{g'(R)|B(R)||S(R)|^{\frac{4}{3}}}(\int_{\Omega}\frac{\lambda'g}{\lambda}-\int_{B(R)}\frac{\lambda'g}{\lambda})\\
    \geq & \frac{1}{|S(R)|^{\frac{4}{3}}}-\frac{1}{g'(R)|B(R)||S(R)|^{\frac{4}{3}}}(g'(R)|B(R)|-\int_{\Omega}((g')^2+\frac{3 g^2}{\lambda^2}))\\
    = & \frac{1}{g'(R)|B(R)||S(R)|^{\frac{4}{3}}}\int_{\Omega}((g')^2+\frac{3 g^2}{\lambda^2}).
\end{align*}

Therefore, we conclude that
\begin{align*}
    |\Sigma|^{\frac{2}{3}}\sigma_1(\Omega)
    \leq \frac{\int_{\Omega}(g')^2+\frac{3 g^2}{\lambda^2}}{g(R)^{\frac{1}{3}}|B(R)|^{\frac{5}{3}}h(\int_{\Omega}\frac{\lambda'g}{\lambda})^{\frac{4}{3}}}
    \leq \frac{g'(R)|B(R)||S(R)|^{\frac{4}{3}}}{g(R)^{\frac{1}{3}}|B(R)|^{\frac{5}{3}}}=|S(R)|^{\frac{2}{3}}\sigma_1(B(R)).
\end{align*}

\end{proof}
\end{lem}

The remaining proof of Theorem \ref{main thm} for $n=4$ follows from verbatim repetition of the case $n\geq 5$. Additionally, the equality is determined by isoperimetric inequality. We complete the proof of Theorem \ref{main thm}. \qed

\begin{rmk}
Due to the limitation of quantities, we cannot get the whole results for $n=3$. However, in that case, we have
\begin{align*}
    \lim_{r\to 0}\frac{\lambda\lambda'g'}{g}=1,\qquad \lim_{r\to\infty}\frac{\lambda\lambda'g'}{g}=+\infty.
\end{align*}
Thus if we choose $R_0$ satisfies 
\begin{align}
    \left.\frac{\lambda\lambda'g'}{g}\right|_{r=2R_0}=\frac{5}{2},
\end{align}
and assume the circumscribed radius of $\Omega$ is no larger than $R_0$. Let $O'$ be the center of the circumscribed geodesic ball of $\Omega$. By the choice of the center of mass, $O$ lies within the circumscribed geodesic ball. Then for any $x\in \Omega$,
\begin{align*}
    d(x,O)\leq d(x,O')+d(O,O')\leq 2R_0.
\end{align*}
Therefore, Lemma \ref{lem-n=5} holds for such $\Omega$ when $n=3$. We only need to repeat the procedures in subsection \ref{subsection4.3} to conclude the following proposition.
\begin{prop}\label{main thm-n=3}
There exists $R_0>0$, such that for any smooth bounded domain $\Omega$ in $\mathbb{H}^3$ with star-shaped mean convex boundary $\partial\Omega$ and circumscribed radius no larger than $R_0$. If $\Omega^{\ast}$ is a geodesic ball with $|\partial\Omega^*|=|\partial\Omega|$, then
\begin{equation*}
    \sigma_1(\Omega)\leq \sigma_1(\Omega^{\ast})
\end{equation*}
and equality holds if and only if $\Omega$ is a geodesic ball.
\end{prop}
We believe that $R_0$ in the proposition can be $+\infty$, i.e. Theorem \ref{main thm} also holds for $n=3$.
\end{rmk}

\subsection{Sketch proof of Corollary \ref{main cor}}$\ $

As in \cite{Ver21}, after choosing a new center of mass, using variational characterization of $\sigma_i(\Omega),1\leq i\leq n$ and substituting terms, it holds that
\begin{align}\label{eq-cor-1}
\int_{\Sigma}g(r)^2d\mu\leq \frac{1}{n-1}\sum_{i=1}^{n-1}\frac{1}{\sigma_i(\Omega)}\int_{\Omega}((g')^2+(n-1)\frac{g^2}{\lambda^2} ) dv.
\end{align}
We have already shown in the proof of Theorem \ref{main thm} that
\begin{align}\label{eq-cor-2}
\frac{\int_{\Omega}((g')^2+(n-1)\frac{g^2}{\lambda^2}) dv}{\int_{\Sigma}g^2 d\mu}\leq \sigma_1(\Omega^{\ast}).
\end{align}
So Corollary \ref{main cor} follows from \eqref{eq-cor-1}, \eqref{eq-cor-2} and $\sigma_1(\Omega^{\ast})=\sigma_2(\Omega^{\ast})=\cdots=\sigma_n(\Omega^{\ast})$.\qed


\begin{bibdiv}
\begin{biblist}

\bib{AS96}{article}{
   author={Aithal, A. R.},
   author={Santhanam, G.},
   title={Sharp upper bound for the first non-zero Neumann eigenvalue for
   bounded domains in rank-$1$ symmetric spaces},
   journal={Trans. Amer. Math. Soc.},
   volume={348},
   date={1996},
   number={10},
   pages={3955--3965},
}


\bib{AV22}{article}{
   author={Anoop, T. V.},
   author={Verma, Sheela},
   title={Szeg\"{o}-Weinberger type inequalities for symmetric domains in simply
   connected space forms},
   journal={J. Math. Anal. Appl.},
   volume={515},
   date={2022},
   number={2},
   pages={Paper No. 126429, 16},
}
		


\bib{Bre13}{article}{
   author={Brendle, Simon},
   title={Constant mean curvature surfaces in warped product manifolds},
   journal={Publ. Math. Inst. Hautes \'{E}tudes Sci.},
   volume={117},
   date={2013},
   pages={247--269},
}

\bib{Bro01}{article}{
   author={Brock, F.},
   title={An isoperimetric inequality for eigenvalues of the Stekloff
   problem},
   journal={ZAMM Z. Angew. Math. Mech.},
   volume={81},
   date={2001},
   number={1},
   pages={69--71},
}

\bib{BFNT21}{article}{
   author={Bucur, Dorin},
   author={Ferone, Vincenzo},
   author={Nitsch, Carlo},
   author={Trombetti, Cristina},
   title={Weinstock inequality in higher dimensions},
   journal={J. Differential Geom.},
   volume={118},
   date={2021},
   number={1},
   pages={1--21},
}



\bib{BS14}{article}{
   author={Binoy},
   author={Santhanam, G.},
   title={Sharp upperbound and a comparison theorem for the first nonzero
   Steklov eigenvalue},
   journal={J. Ramanujan Math. Soc.},
   volume={29},
   date={2014},
   number={2},
   pages={133--154},
}


\bib{CGGS24}{article}{
   author={Colbois, Bruno},
   author={Girouard, Alexandre},
   author={Gordon, Carolyn},
   author={Sher, David},
   title={Some recent developments on the Steklov eigenvalue problem},
   journal={Rev. Mat. Complut.},
   volume={37},
   date={2024},
   number={1},
   pages={1--161},
}


\bib{Esc99}{article}{
   author={Escobar, Jos\'{e} F.},
   title={An isoperimetric inequality and the first Steklov eigenvalue},
   journal={J. Funct. Anal.},
   volume={165},
   date={1999},
   number={1},
   pages={101--116},
}
	

\bib{FL21}{article}{
   author={Freitas, Pedro},
   author={Laugesen, Richard S.},
   title={From Neumann to Steklov and beyond, via Robin: the Weinberger way},
   journal={Amer. J. Math.},
   volume={143},
   date={2021},
   number={3},
   pages={969--994},
}


\bib{FS11}{article}{
   author={Fraser, Ailana},
   author={Schoen, Richard},
   title={The first Steklov eigenvalue, conformal geometry, and minimal
   surfaces},
   journal={Adv. Math.},
   volume={226},
   date={2011},
   number={5},
   pages={4011--4030},
}


\bib{FS19}{article}{
   author={Fraser, Ailana},
   author={Schoen, Richard},
   title={Shape optimization for the Steklov problem in higher dimensions},
   journal={Adv. Math.},
   volume={348},
   date={2019},
   pages={146--162},
}



\bib{Ger11}{article}{
   author={Gerhardt, Claus},
   title={Inverse curvature flows in hyperbolic space},
   journal={J. Differential Geom.},
   volume={89},
   date={2011},
   number={3},
   pages={487--527},
}



\bib{GP12}{article}{
   author={Girouard, Alexandre},
   author={Polterovich, Iosif},
   title={Upper bounds for Steklov eigenvalues on surfaces},
   journal={Electron. Res. Announc. Math. Sci.},
   volume={19},
   date={2012},
   pages={77--85},
}


\bib{GP17}{article}{
   author={Girouard, Alexandre},
   author={Polterovich, Iosif},
   title={Spectral geometry of the Steklov problem (survey article)},
   journal={J. Spectr. Theory},
   volume={7},
   date={2017},
   number={2},
   pages={321--359},
}



\bib{Kar17}{article}{
   author={Karpukhin, Mikhail},
   title={Bounds between Laplace and Steklov eigenvalues on nonnegatively
   curved manifolds},
   journal={Electron. Res. Announc. Math. Sci.},
   volume={24},
   date={2017},
   pages={100--109},
}




\bib{KV14}{article}{
   author={Kokarev, Gerasim},
   title={Variational aspects of Laplace eigenvalues on Riemannian surfaces},
   journal={Adv. Math.},
   volume={258},
   date={2014},
   pages={191--239},
}


		
\bib{KW23}{article}{
   author={Kwong, Kwok-Kun},
   author={Wei, Yong},
   title={Geometric inequalities involving three quantities in warped
   product manifolds},
   journal={Adv. Math.},
   volume={430},
   date={2023},
   pages={Paper No. 109213, 28},
}



\bib{Sch43}{article}{
   author={Schmidt, Erhard},
   title={Beweis der isoperimetrischen Eigenschaft der Kugel im
   hyperbolischen und sph\"{a}rischen Raum jeder Dimensionenzahl},
   language={German},
   journal={Math. Z.},
   volume={49},
   date={1943},
   pages={1--109},
}


\bib{Ver21}{article}{
   author={Verma, Sheela},
   title={An isoperimetric inequality for the harmonic mean of the Steklov
   eigenvalues in hyperbolic space},
   journal={Arch. Math. (Basel)},
   volume={116},
   date={2021},
   number={2},
   pages={193--201},
}



\bib{Wei54}{article}{
   author={Weinstock, Robert},
   title={Inequalities for a classical eigenvalue problem},
   journal={J. Rational Mech. Anal.},
   volume={3},
   date={1954},
   pages={745--753},
}

\bib{Wei54-2}{article}{
   author={Weinstock, Robert},
   title={Inequalities for a classical eigenvalue problem},
   journal={J. Rational Mech. Anal.},
   volume={3},
   date={1954},
   pages={745--753},
}

\bib{Wei56}{article}{
   author={Weinberger, H. F.},
   title={An isoperimetric inequality for the $N$-dimensional free membrane
   problem},
   journal={J. Rational Mech. Anal.},
   volume={5},
   date={1956},
   pages={633--636},
}

\bib{YY17}{article}{
   author={Yang, Liangwei},
   author={Yu, Chengjie},
   title={A higher dimensional generalization of Hersch-Payne-Schiffer
   inequality for Steklov eigenvalues},
   journal={J. Funct. Anal.},
   volume={272},
   date={2017},
   number={10},
   pages={4122--4130},
}

		

\end{biblist}
\end{bibdiv}
\end{document}